\documentclass[11pt,a4paper]{amsart}
\usepackage{geometry}                % See geometry.pdf to learn the layout options. There are lots.
\usepackage{amssymb}
\newtheorem*{theorem}{Theorem}%[section]
\newtheorem*{corollary}{Corollary} 
\newtheorem{lemma}{Lemma}

\newtheorem*{proposition}{Proposition}

\def\N{\mathbb N}
\def\Z{\mathbb Z}
\def\<{\langle}
\def\>{\rangle}
\def\Area{{\rm Area}}
\def\e{\varepsilon}
\def\G{\Gamma}

\title{The Dehn functions of $Out(F_n)$ and $Aut(F_n)$}
\begin{document}

\begin{abstract}  For $n\ge 3$,
the Dehn functions of $Out(F_n)$ and $Aut(F_n)$ are exponential. 
Hatcher and Vogtmann proved that they are at most exponential, 
and the complementary lower bound in the case $n=3$ was established
by Bridson and Vogtmann. 
Handel and Mosher completed the proof
by reducing the lower bound for  $n>4$ to the case $n=3$. 
In this note we give a shorter, more direct proof  of this last reduction.

\end{abstract}

% author one information
\author[Bridson]{Martin R.~Bridson}
\address{Martin R.~Bridson\\
Mathematical Institute \\
24-29 St Giles' \\
Oxford OX1 3LB \\
U.K. }
\email{bridson@maths.ox.ac.uk}

% author two information
\author[Vogtmann]{Karen Vogtmann }
\address{Karen Vogtmann\\
Department of Mathematics\\
Cornell University\\
Ithaca NY 14853 }
\email{vogtmann@math.cornell.edu}

\thanks{Bridson is supported by an EPSRC Senior Fellowship.
Vogtmann is supported by NSF grant DMS-0204185.}



\maketitle

%\section{Introduction}
Dehn functions provide upper bounds on the complexity of the word problem in finitely presented groups. They are examples of filling functions: if a group $G$ acts properly and cocompactly on a simplicial complex $X$, then the Dehn function of $G$ is asymptotically equivalent to the function that provides the optimal upper bound on the area of least-area discs in $X$, where the bound is expressed as a function of the length of the boundary of the disc.  This article is concerned with the Dehn functions of automorphism groups of finitely-generated free groups.

Much of the contemporary study of $Out(F_n)$ and $Aut(F_n)$ is based on the deep analogy between these groups, mapping class groups, and lattices in semisimple Lie groups, particularly 
${\rm{SL}}(n,\Z)$. The Dehn functions of mapping class groups are quadratic \cite{mosher}, 
as is the Dehn function of 
${\rm{SL}}(n,\Z)$ if $n\ge 5$ (see \cite{young}). In contrast, Epstein {\em{et al.}}
\cite{Ep} proved that the Dehn function of ${\rm{SL}}(3,\Z)$ is exponential. Building on their result, we proved in \cite{BV} that $Aut(F_3)$ and $Out(F_3)$ also have exponential Dehn functions.  Hatcher and Vogtmann \cite{HV} established an exponential upper bound on the Dehn function of $Aut(F_n)$ and $Out(F_n)$ for all $n\ge 3$. The comparison with ${\rm{SL}}(n,\Z)$ might lead one to suspect that this last result is not optimal for 
large $n$, but recent work of Handel and Mosher \cite{HM} shows that in fact it is:  they
establish an exponential lower bound by using
their general results on quasi-retractions to reduce to the case $n=3$.

\begin{theorem}\label{t:main} For $n\ge 3$, the Dehn functions of $Aut(F_n)$ and $Out(F_n)$ are 
exponential.
\end{theorem}
 
 This theorem answers Questions 35 and 37 of \cite{BVsurvey}.

 We learned the contents of \cite{HM} from Lee Mosher at Luminy in June 2010 and realized that
one can also reduce the Theorem %\ref{t:main} 
to the case $n=3$ using
a simple observation about natural maps between different-rank  Outer spaces and Auter spaces (Lemma \ref{maps}). 
The purpose of this
note is record this observation and the resulting proof of the Theorem.% \ref{t:main}.    

\subsection*{1. Definitions}

Let $A$ be a 1-connected simplicial complex. We consider simplicial
loops $\ell\colon S\to A^{(1)}$, where $S$ is a simplicial subdivision of the
circle. A {\em simplicial filling} of $\ell$ is a simplicial map $L\colon D\to A^{(2)}$,
where $D$ is a triangulation of the 2-disc and $L|_{\partial D}=\ell$.
Such fillings always exist, by simplicial approximation.
The filling area of $\ell$, denoted
$\Area_A(\ell)$, is the least number of triangles in the domain of any simplicial filling of $\ell$. 
The {\em{Dehn function}}\footnote{The standard definition of
area and Dehn function are
phrased in terms of singular discs, but this version is $\simeq$ equivalent.}
 of $A$ is the least function 
$\delta_A\colon \N\to\N$ such that $\Area_A(\ell)\le \delta_A(n)$ for all loops
of length $\le n$ in $A^{(1)}$. The Dehn function of a finitely presented group $G$ is the Dehn function
of any 1-connected 2-complex on which $G$ acts simplicially with finite stabilizers and compact quotient. This is well-defined up
to the following equivalence relation:  functions $f,g\colon \N\to\N$ are equivalent if $f\preceq g$ and $g\preceq f$,
where $f\preceq g$ means that  there is a 
constant $a>1$ such that $f(n) \le a\, g(an+a) + an +a$. The Dehn function can be interpreted as a measure of the
complexity of the word problem for $G$ --- see \cite{mrb-bfs}.

\begin{lemma}\label{loops} If $A$ and $B$ are 1-connected simplicial complexes, $F\colon A \to B$ is a simplicial map,  and $\ell$ is a loop in the 1-skeleton of $A$, then $\Area_A(\ell)\geq \Area_B(F\circ\ell)$.
\end{lemma}

\begin{proof} If $L\colon D\to A$ is a simplicial filling of $\ell$, then
$F\circ L$ is a simplicial filling of $F\circ \ell$, with the same
number of triangles in the domain $D$.
\end{proof}

\begin{corollary}\label{c:factor} Let $A, B$ and $C$ be 1-connected
simplicial complexes with simplicial maps $A\to B\to C$. 
Let $\ell_n$ be a sequence of simplicial loops in $A$ whose length is bounded above by a linear function of $n$, let $\overline \ell_n$ be the image loops in $C$ and let $\alpha(n) = \Area_C(\overline \ell_n)$. Then the Dehn function of $B$ satisfies
$\delta_B(n)\succeq\alpha(n)$.
\end{corollary}

\begin{proof} This follows from Lemma~\ref{loops} together with the observation that a simplicial map 
does not increase the length of any loop in the 1-skeleton.   
\end{proof}

\subsection*{2. Simplicial complexes associated to $Out(F_n)$ and $Aut(F_n)$.}  
Let $K_n$ denote the spine of  Outer space, as defined in \cite{CV}, and $L_n$ the spine of  Auter space, as defined in \cite{HV}.
These are contractible simplicial complexes with cocompact proper actions by $Out(F_n)$ and $Aut(F_n)$  respectively, so we may use them to compute the Dehn functions for these groups.  

Recall from \cite{CV} that a {\em marked graph} is a finite metric graph $\G$ together with a homotopy equivalence $g\colon R_n \to \G$, where $R_n$ is a fixed graph with one vertex and $n$ loops.  A vertex of 
$K_n$ can be represented either as a marked graph $(g,\G)$ with all vertices of valence at least three, or as a free minimal action of $F_n$ on a simplicial tree (namely the universal cover of $\G$). A vertex of $L_n$ has the same descriptions except that there is a chosen basepoint in the marked graph (respected by the marking) or in the simplicial tree.   Note that we allow  marked graphs to have separating edges.   Both $K_n$ and $L_n$ are flag complexes, so to define them it suffices to describe what it means for vertices to be adjacent.   In the marked-graph description, vertices of  $K_n$ (or $L_n$) are adjacent if  one can be obtained  from the other by a forest collapse (i.e. collapsing each component of a forest to a point).  

 \subsection*{3. Three Natural Maps}

There is a  {\em{forgetful map}} $\phi_n\colon L_n\to K_n$ which simply forgets the basepoint; this map is simplicial. 

Let $m<n$. We fix an ordered basis for $F_n$, identify $F_m$ with the subgroup generated
by the first $m$ elements of the basis, and identify $Aut(F_m)$ with the
subgroup of $Aut(F_n)$ that fixes the last $n-m$ basis elements. We consider
two maps associated to this  choice of basis.

First, there is
an equivariant {\em{augmentation map}}  $\iota\colon L_m\to L_n$ which attaches a bouquet of $n-m$ circles to the basepoint of each marked graph and marks them with the last $n-m$ basis elements of $F_n$.  This map is  simplicial, since a forest collapse has no effect on the bouquet of circles at the basepoint.  
 
Secondly, there is a {\em{restriction map}} $\rho\colon K_n\to K_m$ which is easiest to describe using trees.  A point in $K_n$ is given by a minimal free simplicial action of $F_n$ on a tree $T$ with no vertices of valence 2.  We define $\rho(T)$ to be the minimal invariant
subtree for $F_m<F_n$; more explicitly, $\rho(T)$ is the union of the axes in $T$ of all elements of $F_m$. (Vertices of $T$ that have valence 2 in $\rho(T)$
are no longer considered to be vertices.)

One can also describe $\rho$ in terms of marked graphs. The chosen embedding $F_m<F_n$ corresponds to choosing an $m$-petal subrose $R_m\subset R_n$.  
A vertex in $K_n$ is given by a graph $\G$ marked with a homotopy equivalence
 $g\colon  R_n\to \G$, and the restriction of $g$ to $R_m$ lifts to a homotopy
equivalence $\widehat g\colon  R_m\to \widehat \G$, where $\widehat \G$ is
the covering space corresponding to $g_*(F_m)$.   There is a canonical retraction $r$ of $\widehat \G$ onto its  {\em compact core}, i.e.~the smallest connected subgraph containing all nontrivial embedded loops in $\G$.
 Let $\widehat  \G_0$ be the graph obtained by erasing all vertices
of valence 2 from the compact core and define  $\rho(g,\G)=(r\circ \widehat g, \widehat \G_0)$.

\begin{lemma} For $m<n$, the restriction map $\rho\colon K_n\to K_m$ is simplicial.
\end{lemma}

\begin{proof} Any forest collapse in $\G$ is covered by
a forest collapse in $\widehat \G$ that preserves the compact core, so $\rho$ preserves adjacency.
\end{proof}

 \begin{lemma}\label{maps} For $m<n$, the following diagram of simplicial maps commutes:
$$ 
\begin{matrix}
L_m&\buildrel{\iota}\over\to &L_{n}\\
\phi_m\downarrow&&\downarrow\phi_n\\
K_m&\buildrel{\rho}\over\leftarrow & K_{n}
\end{matrix}
$$
\end{lemma}

\begin{proof}
Given a marked graph with basepoint $(g,\G;v)\in L_n$, the marked graph
$\iota(g,\G;v)$ is obtained by attaching $n-m$ loops at $v$ labelled
by the elements $a_{m+1},\dots,a_n$
of our fixed basis for $F_n$. Then $(g_n,\G_n):=\phi_n\circ\iota(g,\G;v)$ is obtained by forgetting the basepoint, and the cover of $(g_n,\G_n)$ corresponding to
$F_m<F_n$ is obtained from a copy of $(g,\G)$ (with its labels) by attaching
$2(n-m)$ trees. (These trees are obtained from the Cayley graph of $F_n$
as follows: one cuts  at an edge labelled $a_i^\e$, with $i\in\{m+1,\dots,n\}$ and $\e=\pm 1$, takes one component of the result, and then attaches the hanging
edge
to the basepoint $v$ of $\G$.) The effect of $\rho$ is to 
delete these trees. 
\end{proof}

\subsection*{4. Proof of the Theorem}  In the light of the Corollary
 %\ref{c:factor} 
 and Lemma \ref{maps},
it suffices to exhibit a sequence of loops $\ell_i$ in the 1-skeleton of $L_3$
whose lengths are bounded by a linear function of $i$ and whose filling area
when projected to $K_3$ grows exponentially as a function of $i$. Such a 
sequence of loops is essentially described in \cite{BV}.   What we actually
described there were words in the generators of $Aut(F_3)$ rather than
loops in $L_3$, but standard quasi-isometric arguments show that this is equivalent. More explicitly, the words we considered were  $w_i=T^iAT^{-i}BT^iA^{-1}T^{-i}B^{-1}$
where 
 \[
T\colon\begin{cases} 
            a_1\mapsto a_1^2a_2\cr
            a_2\mapsto a_1a_2\cr
            a_3\mapsto a_3
            \end{cases}
 A\colon\begin{cases} 
            a_1\mapsto a_1\cr
            a_2\mapsto a_2\cr
            a_3\mapsto a_1a_3
            \end{cases}
B\colon\begin{cases} 
            a_1\mapsto a_1\cr
            a_2\mapsto a_2\cr
            a_3\mapsto a_3a_2
            \end{cases} 
 \]         
To interpret these as loops in the 1-skeleton of $L_3$ (and $K_3$) we note that $A=\lambda_{31}$ and $B=\rho_{32}$ are elementary transvections and $T$ is the composition of two elementary transvections:  $T=\lambda_{21}\circ \rho_{12}$.  Thus $w_i$ is the product of $8i+4$ elementary transvections.  There is a
 (connected) subcomplex of the 1-skeleton of  $L_3$ spanned by roses 
(graphs with a single vertex) and Nielsen graphs (which have $(n-2)$ loops at the base vertex
and a further trivalent vertex).  We say roses are adjacent if they have distance $2$ in this graph.  

Let $I\in L_3$ be the rose  marked by the identity map $R_3\to R_3$.  Each elementary transvection $\tau$ moves $I$ to an adjacent rose $\tau I$, which is connected to $I$ by  a Nielsen graph $N_\tau$.   A composition $\tau_1\ldots\tau_k$ of elementary transvections gives a  path through adjacent roses $I, \tau_1I, \tau_1\tau_2I, \ldots,\tau_1\tau_2\ldots \tau_kI$; the Nielsen graph connecting $\sigma I$ to $\sigma\tau I$ is $\sigma N_\tau$.   Thus the word $w_i$ corresponds to a loop $\ell_i$  of length $16i+8$ in the 1-skeleton of $L_3$.  
Theorem A of \cite{BV} provides an exponential lower bound on the filling area of $\phi\circ \ell_i$ 
in $K_3$. \qed
%\section{}
%\subsection{}
 
\smallskip
The square of maps in Lemma \ref{maps} ought to have many uses beyond the one in this note (cf.~\cite{HM}). We mention just one, for illustrative purposes. This is a special case of the fact that every infinite cyclic subgroup of $Out(F_n)$ is quasi-isometrically embedded \cite{alibegovic}.

\begin{proposition}
The cyclic subgroup of $Out(F_n)$ generated by any Nielsen transformation (elementary transvection) is quasi-isometrically embedded.
\end{proposition}

\begin{proof}  Each Nielsen transformation is in the image of
the map $\Phi\colon  Aut(F_2)\to Aut(F_n)\to Out(F_n)$ given by the inclusion of a 
free factor $F_2<F_n$. Thus it suffices to prove that if a
cyclic subgroup $C=\<c\><Aut(F_2)$ has infinite image in $Out(F_2)$, then
$t\mapsto \Phi(c^t)$ is a quasi-geodesic. This is 
equivalent to the assertion that some (hence any) $C$-orbit  in $K_n$ is quasi-isometrically embedded,
where $C$ acts on $K_n$ as $\Phi(C)$ and $K_n$ is given the piecewise Euclidean metric where all edges
have length $1$.

$K_2$ is a tree and $C$ acts on $K_2$ as a hyperbolic isometry, so the $C$-orbits in
$K_2$ are quasi-isometrically embedded. For each  $x\in L_2$,  the $C$-orbit of $\phi_2(x)$
is the image of the quasi-geodesic
%$\Z\to L_2\to K_2$ given by 
$t\mapsto c^t.\phi_2(x) =\phi_2(c^t.x)$.  We factor $\phi_2$  
as a composition of $C$-equivariant simplicial maps $L_2\overset{\iota}\to K_n\overset{\phi_n}\to
 K_2$, as in Lemma \ref{maps},
to deduce that  the $C$-orbit
of $\phi_n\iota(x)$ in $K_n$ is
quasi-isometrically embedded.
\end{proof}

A slight variation on the above argument shows that if one lifts a free group
of finite index $\Lambda<Out(F_2)$ to $Aut(F_2)$ and then maps it to 
$Out(F_n)$ by choosing a free factor $F_2<F_n$, then the inclusion   
$\Lambda\hookrightarrow Out(F_n)$ will be a quasi-isometric embedding.

%\bibliography{dehn}{}
%\bibliographystyle{plain}
%\end{document}

\end{document}